\documentclass{amsart}

\usepackage{amsmath,amscd,amssymb,amsthm,galois}
\usepackage{latexsym,enumerate,graphicx,geometry,hyperref,rotating,braket,mathrsfs}
\usepackage[all]{xy}

\DeclareMathOperator{\Gal}{Gal}
\newtheorem{def_}{Definition}[section]
\newtheorem{lemma}[def_]{Lemma}
\newtheorem{theorem}[def_]{Theorem}
\newtheorem{corollary}[def_]{Corollary}
\newtheorem{proposition}[def_]{Proposition}
\newtheorem*{remark}{Remark}
\newtheorem{def_thm}[def_]{Theorem and Definition}

\newcommand{\zz}{\mathbb{Z}}

\newcommand{\Q}{\mathbb{Q}}

\newcommand{\ra}{\rightarrow}

\newcommand{\Pic}{\mbox{Pic}}
\newcommand{\Cl}{\mbox{Cl}}
\newcommand{\mmod}{\mbox{mod}~}

\newcommand{\ok}{\mathfrak{o}_K}
\newcommand{\oo}{\mathfrak{o}}
\newcommand{\so}{\mathscr{O}}
\newcommand{\ord}{\mathfrak{o}}
\newcommand{\aaa}{\mathfrak{a}}
\newcommand{\p}{\mathfrak{p}}
\newcommand{\hp}{\mathfrak{P}}
\newcommand{\pp}{\mathfrak{P}}

\newcommand{\A}{\mathfrak{A}}
\newcommand{\B}{\mathfrak{B}}

\newcommand{\f}{\mathfrak{f}}
\newcommand{\m}{\mathfrak{m}}

\begin{document}

\title{On Ring Class Fields of Number Rings}
\author[H. Yi]{Hairong Yi}

\address{State Key Laboratory of Information Security\\
Institute of Information Engineering\\
Chinese Academy of Sciences, Beijing, 100093, China.\newline
\indent School of Cyber Security\\
University of Chinese Academy of Sciences\\
Beijing, 100049, China}
\email{yihairong@iie.ac.cn}
\thanks{The work was supported by National Natural Science Foundation of China (Grant No. 11701552)}
\author[C. Lv]{Chang Lv}
\address{State Key Laboratory of Information Security\\
Institute of Information Engineering\\
Chinese Academy of Sciences, Beijing, 100093, China}
\email{lvchang@iie.ac.cn}
\subjclass[2010]{Primary 11R37, 11Y40}
\keywords{Number rings, Picard groups, ring class fields, norm form equations}
\date{April 28, 2018}

\begin{abstract}
For a number field $K$, we extend the notion of the ring class field of an order in $K$ [C. Lv and Y. Deng, SciChina. Math., 2015] to that of an arbitrary number ring in $K$. We give both ideal-theoretic and idele-theoretic description of this number ring class field, and characterize it as a subfield of the ring class field of some order.  As an application, we use it to give a criterion of the solvability of a higher degree norm form equation over a number ring and finally describe algorithms to compute this field.



\end{abstract}
\maketitle

\section{Introduction}
Ring class fields are  closely related to the problem of the existence of integral solutions of norm form equations, see for example \cite{cox,  integralsquare,dioph.equa.,multi_norm}. The original and typical application can be found in \cite{cox} written by David A. Cox, which introduced how to use the ring class field to give a criterion of the integral solvability of the quadratic form equation $p=x^2+ny^2$ over $\zz$.
In \cite{multi_norm} Wei and Xu proposed a rather different idea to give an idelic form criterion of the integral solvability of the quadratic equation $a=x^2\pm dy^2$ over $\zz$ for an arbitrary integer $a$.
The ring class fields considered in \cite{cox,multi_norm} were all restricted to the case of orders of imaginary quadratic fields \cite{invitation,construct}, one described in ideal-theoretic form and the other in idele-theoretic form. Then C. Lv and Y. Deng \cite{Chang Lv} generalized this notion to an order of an arbitrary number field. They gave an explicit ideal-theoretic description of this ring class field and applied it to the solvability of the diophantine equation $p=x^2+ny^2$ over $\oo_F$ where $F$ is an imaginary quadratic field.

In this article, we consider to extend this notion more generally to an arbitrary number ring $R$ of an arbitrary number field $K$. We start from the explicit ideal-theoretic description of the ring class field of $R$, including the determination of its modulus and congruence subgroup, and then find and verify the idele-theoretic correspondence of it, of which the form has already been known in the case of an order, but its equivalence to the ideal-theoretic version has not been verified in literature yet. By the way we give a characterization of the ring class field of a number ring in terms of the ring class field of an order. As an application, we extend the main results in \cite{cox} and \cite{multi_norm} to the relative extension case for a higher degree norm form equation over a number ring. Finally, we introduce how to compute this ring class field using Pari/GP \cite{pari}.

The article is arranged as follows. In Section 2 we present the explicit ideal-theoretic description of the ring class field of a number ring. In Section 3 we give a characterization of such a kind of class fields. In Section 4 we give and verify the idele-theoretic correspondence of number ring class fields. Section 5 discusses further the structure of number rings and proposes an application and Section 6 describes the algorithms for computing number ring class fields using Pair/GP.


\section{The Ring Class Field Defined By a Number Ring}\label{theory}
  Given a number field $K$, a number ring $R$ in $K$ is an integral domain for which the field of fractions is equal to $K$ \cite{arith_nr}. For example in $\Q$ there are infinitely many number rings, such as $\zz$, $\zz[\frac{1}{2},\frac{1}{3}]$ and $\Q$ itself. $R$ has its integral closure in $K$, which is denoted by $\so$. The conductor of $R$ is defined to be $\f:=\set{x\in K|x\so\subseteq R}$, as in the case of orders.

In this section, first we recall some basic facts about the Picard group of $R$ and the generalized ideal class group. Then we focus on a subgroup of the Picard group of $R$ whose elements are prime to $\f$, and found it naturally determining a congruence subgroup. By showing that this subgroup is indeed the whole Picard group of $R$, and through the class field theory, the Picard group of $R$ corresponds to a class field of $K$, which is exactly the ring class field of $R$.

\subsection{Basic Notation and Results}\label{basic-notation-and-results}
Denote by $\ok$ the ring of integers of $K$. Let $\m$ be a modulus of $K$, in other words $\m=\m_0\m_\infty$ is a formal product of an integral ideal of $K$ and a subset of the set of real places of $K$. So an integral $\ok$-ideal can be viewed as a modulus of which the infinity part is 1 by default. Let $J^\m_K$ be the group of all fractional $\ok$-ideals relatively prime to $\m_0$ and let $P^\m_{K,1}$ be the subgroup of $J^\m_K$ generated by the principal ideals $\alpha\ok$, where $\alpha\in \ok$ satisfies
$$\alpha\equiv 1~\mmod \m_0\quad \mbox{and}\quad \sigma(\alpha)>0~~\mbox{for every real place $\sigma$ dividing $\m_\infty$}. $$

\begin{def_}\cite{janusz}
The quotient $\Cl^\m_K:=J^\m_K/P^\m_{K,1}$ is called the ray class group for $\m$. A basic fact is that it is a finite group. A subgroup $H^\m \subseteq J^\m_K$ is called a congruence subgroup for $\m$ if it satisfies $P^\m_{K,1}\subseteq H^\m \subseteq J^\m_K$, and the quotient $J^\m_K/ H^\m$ is called a generalized ideal class group for $\m$. Obviously it is also a finite group.
\end{def_}

For a number ring $R$ of $K$, we consider the group of all invertible ideals of $R$, denoted by $J(R)$. Those principal ones, i.e. $\alpha R$ for some $\alpha\in K^*$\footnote{For a ring $A$, $A^*$ denotes the subset of multiplicative invertible elements of $A$.}, form a subgroup of $J(R)$, denoted by $P(R)$. The quotient $\Pic(R):= J(R)/P(R)$ is called the \emph{Picard group} of $R$. In the case $R=\ok$, $\Pic(R)$ is exactly the class group $\Cl_K$ of $K$.

For the properties of number rings, we need the following results. See detailed proofs in \cite{arith_nr}.

\begin{proposition}\label{basic prop. of number ring}\label{basic prop. of n.r 2}
Let $R,K,\so,\f$ be as above. We have
\begin{itemize}
  \item[(1)] For any prime ideal $\hp$ of $R$,
  $$\hp~\mbox{is invertible} \Leftrightarrow R_\hp~\mbox{is a discrete valuation ring} \Leftrightarrow \hp\nmid \f~(\text{i.e.}~\pp\nsupseteq\f).$$
  \item[(2)] Let $\mathfrak{A}$ be an integral $R$-ideal prime to the conductor $\f$, i.e. $\mathfrak{A}+\f=R$. Then $\mathfrak{A}$ is invertible and factors uniquely into a product of invertible prime ideals of $R$.
  \item[(3)] $\so$ is a Dedekind domain, and there exists a set $T$ of prime ideals of $\ok$ such that
  $$\so=R\ok=\oo_{K,T}=\set{x\in K|\mathrm{ord}_\p(x)\geq 0~\mbox{for all}~\p\notin T}.$$
  \item[(4)] There is an injective map $\set{\text{prime ideals of }\so}\hookrightarrow \set{\mbox{prime ideals of }\ok}$ by the natural contraction of ideals, and the missing part is exactly the set of the prime ideals in $T$. Moreover, we have $\p\so=\so$ for any $\p\in T$.
\end{itemize}
\end{proposition}

Using symbols in the above proposition we have
$$T=\set{\mbox{prime ideals $\p$ of }\ok|\exists~x\in R \mbox{ s.t. }\mathrm{ord}_\p(x)< 0}.$$
For our goal, we make the following assumption in this paper:\\
\\
\noindent\textbf{Assumption}: The associated set $T$ of $R$ is finite.

\subsection{The Congruence Subgroup Determined by $\Pic(R,\f)$}\label{Sec-congruence-subgroup}
We follow some of the notation and symbols used in \cite{Chang Lv}.
For an ideal $\mathfrak{I}$ of $R$ (or $\ok$, $\so$), consider the set of prime ideals that are ``prime to $\mathfrak{I}$'': set $S(R,\mathfrak{I}):=\set{\mbox{prime ideals $\mathfrak{P}$ of }R| \mathfrak{P}+\mathfrak{I}=R}$. It follows from Proposition \ref{basic prop. of n.r 2} that $\so$ is the ring of $T$-integers for some finite set $T$ of finite places of $\ok$. Let
$$\m_0:=\prod_{\p\in T}\p,\quad \tilde\f=\f\cap\ok.$$
Note that $\f$ is an ideal of both $\so$ and $R$, and $\tilde\f$ is prime to $\m_0$ by Proposition \ref{basic prop. of number ring} (4).

\begin{lemma}
There is a bijection between $S(R,\f)$ and $S(\so,\f)$ by the extension and contraction of ideals. Also there is a bijection between $S(\ok,\tilde{\f}\m_0)$ and $S(\so,\f)$ in the same way.
\end{lemma}
\begin{proof}
 The first bijection is analogous to the case of orders together with the ring of integers, as in \cite[Lemma 3.3]{Chang Lv}, where we only need the fact that if $\pp$ is a prime ideal of $R$ and prime to $\f$, then $\pp\so$ is a prime ideal of $\so$ and prime to $\f$. Since $\pp+\f=R$, there exists $\alpha\in \pp$ and $f\in\f$ such that $\alpha+f=1$. Then for any $x\in\so$, $x=x\alpha+xf\equiv xf ~\mmod \pp\so$ where $xf\in \f\so\subseteq R$. This shows $R/\pp=\so/\pp\so$, which proves the fact. For the second bijection, it is a direct result of Proposition \ref{basic prop. of n.r 2} (4).
\end{proof}

\begin{remark}
We can compose these two bijections to obtain a bijection between $S(R,\f)$ and $S(\ok,\tilde{\f}\m_0)$: $\pp\mapsto \pp\so \cap \ok$ for any $\pp$ belonging to the left side, and $\p\mapsto\p\so\cap R$ for any $\p$ belonging to the right side.
\end{remark}

Let $J(R,\f)$ (resp. $J(\so,\f)$) be the group generated by $S(R,\f)$ (resp. $S(\so,\f)$). By Proposition \ref{basic prop. of number ring}, it is contained in $J(R)$ (resp. $J(\so)$) and a free abelian group over $S(R,\f)$ (resp. $S(\so,\f)$). These results hold for $J^{\tilde{\f}\m_0}_K$ as well. Therefore we obtain the following

\begin{corollary}\label{J(R,f)}
 There exist isomorphisms $J(R,\f)\cong J(\so,\f) \cong J_K ^{\tilde{\f}\m_0}$. The first isomorphism is given by $\mathfrak{I}\mapsto\mathfrak{I}\so$ for $\mathfrak{I}\in J(R,\f)$ and $\mathfrak{A}'/\mathfrak{B}'\mapsto (\mathfrak{A}'\cap R)/(\mathfrak{B}'\cap R)$ for $\mathfrak{A}', \mathfrak{B}'\subseteq \so$ prime to $\f$. The second isomorphism is given in the same way.
\end{corollary}

We consider the subgroup $P(R,\f)$ of $J(R,\f)$, which is generated by the set of restricted principle $R$-ideals $\check{P}(R,\f):=\set{\alpha R|\alpha\in R,\alpha R+\f=R}\subseteq P(R)$. Here we can define the important intermediate group $\Pic(R,\f):=J(R,\f)/P(R,\f)$. We want to find the correspondence of the subgroup $P(R,\f)$ under the isomorphism given in Corollary \ref{J(R,f)}.

  Let $\check{P}_{K,R}^{\tilde{\f}\m_0}:=\set{\alpha\so\cap \ok|\alpha\in R, \alpha R+\f=R}$ and $\check{P}^{\tilde{\f}\m_0}_{K,1}:=\set{\alpha\ok|\alpha\in\ok, \alpha\equiv 1~(\mmod \tilde{\f}\m_0)}$. Let $P^{\tilde{\f}\m_0}_{K,R}$ be the subgroup of $J_K^{\tilde{\f}\m_0}$ generated by $\check{P}_{K,R}^{\tilde{\f}\m_0}$. Then it is obvious that $P^{\tilde{\f}\m_0}_{K,1}$ is generated by $\check{P}^{\tilde{\f}\m_0}_{K,1}$. It maybe false that $\alpha \so\cap \ok = \alpha\ok$ for $\alpha\in \so$, unless $\alpha\in\ok$ and $\alpha\ok$ is prime to $\m_0$.

\begin{theorem}\label{thm_congruencegp}
With the notation as above, we have
\begin{itemize}
  \item[(1)] The group $P^{\tilde{\f}\m_0}_{K,R}$ is a congruence subgroup for the modulus $\tilde{\f}\m_0$.
  \item[(2)] There is an isomorphism from $\Pic(R,\f)$ to the generalized ideal class group $J_K^{\tilde{\f}\m_0}/P_{K,R}^{\tilde{\f}\m_0}$.
\end{itemize}
\end{theorem}

\begin{proof}
For proving (1) we only need to show $\check{P}_{K,1}^{\tilde{\f}\m_0} \subseteq \check{P}_{K,R}^{\tilde{\f}\m_0}$. Take an element $\alpha \in \ok$ with $\alpha \equiv 1~(\mmod \tilde{\f}\m_0)$. Since $\tilde{\f}\m_0\subseteq \tilde{\f}=\f\cap \ok \subseteq \f\subseteq R$, it follows that $\alpha \in R$. And since $\alpha-1\in \tilde\f$, then $\alpha R+\f=R$. On the other hand, by Corollary \ref{J(R,f)} the equality $\alpha \so\cap \ok=\alpha \ok$ holds for any element $\alpha$ which is prime to $\m_0$.

For (2), in fact the generating set $\check{P}_{K,R}^{\tilde{\f}\m_0}$ has been carefully chosen to be the image of the subset $\check{P}(R,\f)$ under the isomorphism $J(R,\f)\ra J_K ^{\tilde{\f}\m_0}$ in Corollary \ref{J(R,f)}. Therefore the subgroups generated by them are also isomorphic. Then the isomorphism from $\Pic(R,\f)$ to $J_K^{\tilde{\f}\m_0}/P_{K,R}^{\tilde{\f}\m_0}$ is clear.
\end{proof}

\subsection{Isomorphism between $\Pic(R,\f)$ and $\Pic(R)$}
The last step to connect the Picard group of $R$ with a class field of $K$, is to show the isomorphism between groups $\Pic(R,\f)$ and $\Pic(R)$.
Actually, we can show that for any integral $R$-ideal $\mathfrak{g}$ satisfying $\f\mid\mathfrak{g}$, the natural map from $\Pic(R,\mathfrak{g})$ to $\Pic(R)$ is an isomorphism.

\begin{lemma}\label{1}
For an invertible fractional $R$-ideal 
$\A $, if it satisfies $\A R_\pp = R_\pp$ for some prime ideal $\pp$ of $R$, then $\A$ is prime to $\pp$ or equivalently belongs to $ J(R,\pp)$. In other words, $\A=\frac{\aaa}{\mathfrak{b}}$ where $\aaa, \mathfrak{b}$ are integral $R$-ideals that are prime to $\pp$.
\end{lemma}

\begin{proof}
 The condition $\A R_\pp = R_\pp$ implies that $\A \subseteq R_\pp$. Since $\A$ is a finitely generated $R$-module \cite[Section 2]{arith_nr}, there exists some $\alpha \in R\backslash \pp$ such that $\alpha \A\subseteq R$. On the other hand, $\alpha\A R_\pp=\A(\alpha R_\pp)=\A R_\pp=R_\pp$, which means $\alpha\A \nsubseteq \pp$ and hence $\alpha \A +\pp=R$. Let $\aaa=\alpha \A$, $\mathfrak{b}=\alpha R$, and then the proof is done.
\end{proof}

\begin{corollary}\label{1.1}
Let $\B \in J(R)$ be an invertible fractional $R$-ideal, $\pp$ be a prime ideal of $R$. Then there exists some $x \in K^*$ such that $x^{-1}\B$ is prime to $\pp$.
\end{corollary}

\begin{proof}
Since $\B$ is invertible, $\B_\pp$ is a principal $R_\pp$-ideal \cite[Proposition 4.4]{arith_nr}, say $\B_\pp=xR_\pp$ for some $x\in K^*$. Then $(x^{-1}\B) R_\pp=R_\pp$. It follows from Lemma \ref{1} that $x^{-1}\B$ is prime to $\pp$.
\end{proof}

\begin{lemma}\label{2}
Let $S=\set{\pp_1,\cdots,\pp_r}$ be a finite set of prime ideals of $R$, $\B$ be an invertible integral $R$-ideal. Then there exists some $\alpha\in R$, such that
\begin{equation}\label{eq.lem.2}
  \frac{\B}{\alpha}R_{\pp_1}=R_{\pp_1},\qquad  \frac{\B}{\alpha}R_{\pp_i}=\B R_{\pp_i} \quad \mbox{for all }~~2\leq i \leq r.
\end{equation}
\end{lemma}

\begin{proof}
  From the proof of Corollary \ref{1.1}, for any generator $x$ of $\B R_{\pp_1}$ in $R_{\pp_1}$ it follows that $(x^{-1}\B)R_{\pp_1}=R_{\pp_1}$. And since $\B$ is integral we can assume that $x \in \B \subseteq R$.

 \cite[Lemma 5.1]{arith_nr} says that every nonzero ideal of the local number ring $R_\pp$ contains some power of the maximal ideal of $R_\pp$. Assume $e$ to be the exponent such that $\pp_1^e R_{\pp_1}\subseteq \B R_{\pp_1}$.
By the Chinese Remainder Theorem, there exists an $\alpha\in R$ satisfying
\begin{equation*}
  \alpha \equiv x ~\mmod \pp_1^{e+1},~~ \alpha \equiv 1 ~\mmod \pp_i ~~ \mbox{for}~2 \leq i \leq r.
\end{equation*}
We claim that $\alpha$ satisfies (\ref{eq.lem.2}).

Firstly it is obvious that $(\alpha^{-1}\B)R_{\pp_i}=\B R_{\pp_i}$ since $1+\pp_i\subseteq R_{\pp_i}^*$ for all $i \geq 2$. Assume $\alpha=x+\pi$ for some $\pi\in \pp_1^{e+1}$. Then $\pi \in \pp_1\cdot \pp_1^e\subseteq \pp_1 \cdot \B R_{\pp_1}=\pp_1 \cdot xR_{\pp_1}$, which means $\frac{\pi}{x} \in \pp_1 R_{\pp_1}$. Since $1+\pp_1 R_{\p_1}\subseteq R^*_{\pp_1}$, it follows that $\frac{\alpha}{x}=1+\frac{\pi}{x}$ is contained in $R^*_{\pp_1}$. Therefore $\alpha^{-1}\B R_{\pp_1}=x^{-1}\B R_{\pp_1}=R_{\pp_1}$.
\end{proof}

After all these preparation we can prove the main theorem in this section:

\begin{theorem}\label{main isom.}\label{main}
Let notation be as above, and $\mathfrak{g}$ be an integral $R$-ideal divided by $\f$. We have
\begin{itemize}
\item[(1)] The inclusion $J(R,\mathfrak{g})\subseteq J(R)$ induces a monomorphism $\Pic(R,\mathfrak{g})$ $\hookrightarrow \Pic(R)$.
\item[(2)] The monomorphism from $\Pic(R,\mathfrak{g})$ to $\Pic(R)$ is an epimorphism.
\end{itemize}
Then for a general number ring $R$, we have $\Pic(R,\mathfrak{g}) \cong \Pic(R)$.
\end{theorem}

\begin{proof}
For proving injectivity, it is analogous to the case of orders. One may refer to \cite[Lemma 3.9]{Chang Lv}.

For surjectivity, let $S=\set{\mbox{all prime ideals dividing} ~\mathfrak{g}}=\set{\pp_1,\cdots,\pp_r}$. Take an invertible $R$-ideal $\B$, which can be assumed to be integral. It remains to show that there is a representative in the ideal class of $\B$ which is prime to all prime ideals in $S$, i.e. there exists some $\gamma \in K^*$ such that $\gamma^{-1}\B \in J(R,S)$.

By Lemma \ref{2}, for each $i\in \set{1,\cdots,r}$ we find an $\alpha_i \in R$ such that
\begin{equation*}
  \frac{\B}{\alpha_i}R_{\pp_i}=R_{\pp_i},\qquad  \frac{\B}{\alpha_i}R_{\pp_j}=\B R_{\pp_j} \quad \mbox{for all} ~~j \neq i.
\end{equation*}
Let $\gamma=\prod^r_{i=1}\alpha_i$. Then for all $i\in \set{1,\cdots,r}$,
$$\frac{\B}{\gamma}R_{\pp_i}=\frac{\B}{\prod_j \alpha_j}R_{\pp_i}=\frac{1}{\alpha_i}\frac{\B}{\prod_{j\neq i}\alpha_j}R_{\pp_i}=\frac{1}{\alpha_i}\B R_{\pp_i}=R_{\pp_i}.$$
By Lemma \ref{1} it means that $\gamma^{-1}\B$ is prime to all $\pp_i$ in $S$, or equivalently $\gamma^{-1}\B \in J(R,S)$.
\end{proof}

\subsection{The Ring Class Field Defined by $R$}
For readers' convenience, we state here the classic existence theorem of class field theory. See \cite{janusz} and \cite{neukirch_alnt} for standard references.

\begin{theorem}
Let $\m$ be a modulus of $K$, $H^\m$ be a congruence subgroup for $\m$. Then there exists a unique abelian extension $L$ of $K$ such that all primes of $K$ ramified in $L$, finite or infinite, divide $\m$, and if
$$\varphi^\m_{L/K}:J^\m_K\longrightarrow \Gal(L/K)$$
is the Artin map of $L/K$, then $H^\m=\ker\varphi^\m_{L/K}$. Moreover, the Artin reciprocity law holds for $(L,K,\m)$, i.e. $\ker\varphi^\m_{L/K}=P^\m_{K,1}N_{L/K}(J^\m_L)$. Since the Artin map is surjective, we also have
$$J^\m_K/H^\m=J^\m_K/P^\m_{K,1}N_{L/K}(J^\m_L)\cong \Gal(L/K).$$
\end{theorem}

This theorem says that such a pair $(\m,H^\m)$ uniquely determines an abelian extension $L$ of $K$. Choosing the modulus and congruence subgroup as we want, then we obtain

\begin{def_thm}
Let all notation be as above. Then there exists a unique abelian extension $H_R$ of $K$, such that all primes of $K$ ramified in $H_R$ divide $\tilde{\f}\m_0$, and the kernel of the Artin map $\varphi^{\tilde{\f}\m_0}_{H_R/K}:J^{\tilde{\f}\m_0}_K\longrightarrow \Gal(H_R/K)$ is $P^{\tilde{\f}\m_0}_{K,R}$. Moreover, we have the isomorphisms
$$\Pic(R)\cong J_K^{\tilde{\f}\m_0}/P_{K,R}^{\tilde{\f}\m_0}\cong \Gal(H_R/K).$$
We call $H_R$ the \textbf{ring class field} defined by the number ring $R$ (or number ring class field of $R$).
\end{def_thm}

\begin{remark}\label{Artin1}
 We state here a conclusion that will be used in next section: For an unramified prime ideal $\p$ of $\ok$, we have:
\begin{equation*}
  \begin{split}
     \p \text{ splits completely in } H_R &\Leftrightarrow \varphi^{\tilde{\f}\m_0}_{H_R/K}(\p)=1, \\
       &\Leftrightarrow [\p]=[1] \text{ in } J_K^{\tilde{\f}\m_0}/P_{K,R}^{\tilde{\f}\m_0},\\
       &\Leftrightarrow [\p\oo_{K,T}\cap R]=[1]\text{ in } \Pic(R).
   \end{split}
\end{equation*}
These equivalences are direct except the first one, which one may refer to \cite[Corollary 5.21]{cox}.
\end{remark}

\section{A Characterization of Number Ring Class Fields}\label{characterization}
In this section, we give a more explicit characterization of the ring class field defined by a general number ring $R$.

We keep all notation in Section $\S$\ref{theory}. First we assume that the integral closure $\so$ of $R$ is the ring of $T$-integers where $T$ contains only one prime ideal of $K$. We attain our main conclusion as

\begin{theorem}\label{charact. thm}
Assume that $T=\{\p_r\}$. Set $\ord=R\cap\ok$. It is an order in $K$. Denote by $H_\ord$ and $H_R$ the ring class fields defined by $\ord$ and $R$ respectively. Then $H_R$ is a subfield of $H_\ord$. Moreover, it is the invariant subfield of $H_\ord$ fixed by the decomposition group $D_{\p_r}$ of $\p_r$ in the Galois group of $H_\ord/K$.
\end{theorem}

By the class field theory, the crucial point is to find out the relation between their Picard groups.

\begin{lemma}\label{lemma 3.2}
We have an exact sequence
\begin{equation}\label{eq.3.2}
  \prod_{\p\in T} \p^\zz \ra \Pic(\ord)\xrightarrow{\Phi} \Pic(R)\ra 1,
\end{equation}
where the first map is the contraction of prime ideals, and the second map is the natural extension of ideals.
\end{lemma}

\begin{proof}
Recall that $\tilde\f=\f\cap\ok$ where $\f$ is the conductor of $R$. It is divided by the conductor $\f_\ord$ of order $\ord$. For an element $x\in \tilde\f$, $x$ is an algebraic integer and $x\ok\subseteq x\so\cap\ok\subseteq R\cap\ok=\ord$. It follows that $x\in \f_\ord$. Then $\f_\ord\mid \tilde\f$.

By Theorem \ref{main}, $\Pic(\ord)\cong\Pic(\ord,\tilde\f)$ and $\Pic(R)\cong\Pic(R,\f)$. There is also a natural map from $\Pic(\ord,\tilde\f)$ to $\Pic(R,\f)$ by extension of ideals, which is compatible with $\phi$ through these two isomorphisms. We still denote it by $\phi$. It amounts to verifying the following exact sequence
\begin{equation}\label{eq.lem.3.2.1}
  \prod_{\p\in T} \p^\zz \ra \Pic(\ord,\tilde\f)\xrightarrow{\Phi} \Pic(R,\f)\ra 1.
\end{equation}
Recall that $\Pic(\ord,\tilde\f)=J(\ord,\tilde\f)/P(\ord,\tilde\f)$, $\Pic(R,\f)=J(R,\f)/P(R,\f)$. Since $J(\ord,\tilde\f)\cong J(\ok,\tilde\f)$ and $J(R,\f)\cong J(\so,\f)$ by Corollary \ref{J(R,f)}, then following Proposition \ref{basic prop. of number ring} (4) there is an exact sequence
\begin{equation}\label{eq.lem.3.2.2}
  \prod_{\p\in T}\p^\zz\ra J(\ord,\tilde\f)\xrightarrow{\phi} J(R,\f)\ra 1.
\end{equation}
For proving (\ref{eq.lem.3.2.1}) it only remains to show that $\phi^{-1}(P(R,\f))=P(\ord,\tilde\f)$.

Take an element $\alpha R$ in $P(R,\f)$. It has the form $\beta_1 R/\beta_2 R$ where $\beta_i\in R$, $\beta_i R+\f=R$. We may assume that $\alpha\in R$. Recall that the integral closure $\so$ of $R$ is the ring of $T$-integers, and $\so^*$ is the group of $T$-units, usually denoted by $U_T$. Here we need two facts from \cite{arith_nr}: the torsion-free part of $U_T$ is of rank ($r+\#T$) where $r$ is the rank of the torsion-free part of $\ok^*$; and the index $[\so^*:R^*]$ is finite. These two facts imply that there exists an $\epsilon\in R^*$ such that $\gamma=\epsilon \alpha$ satisfies $\mbox{ord}_\p(\gamma)\geq 0$ for all $\p\in T$. Then $\gamma\in \ok\cap R=\ord$ and $\gamma R=\alpha R$. Moreover $\gamma R+\f=R$, which implies $\gamma\so+\f=\so$, and then $\gamma\so\cap\ok+\tilde\f=\ok$. By Proposition \ref{basic prop. of number ring} (4) it follows that $\gamma\ok=(\gamma\so\cap\ok)\prod_{\p\in T}\p^{x_\p}$, and then $\gamma\ok+\tilde\f=\ok$ since prime ideals in $T$ is prime to $\tilde\f$. Therefore $\gamma\ord+\tilde\f=\ord$, and then $\gamma\ord\in P(\ord,\tilde\f)$ is a preimage of $\alpha R$ under the map $\phi$. This completes the proof.
\end{proof}

\begin{proof}[Proof of Theorem 3.1]
By the definition of ring class field, it follows that $\Pic(\ord)\cong\Pic(\ord,\f_\ord)\cong J^{\f_\ord}_K/P^{\f_\ord}_{K,\ord}\cong \Gal(H_\ord/K)$. Recall that the last isomorphism is constructed through the Artin map $\varphi^{\f_\ord}_{H_\ord/K}$. Then Lemma \ref{lemma 3.2} implies that $\Pic(R)\cong J^{\f_\ord}_K/(P^{\f_\ord}_{K,\ord}\cdot\langle\p_r\rangle)$, where we denote by $\langle\p_r\rangle$ the group generated by $\p_r$. Note that the subgroup in $\Gal(H_\ord/K)$ generated by $\varphi^{\f_\ord}_{H_\ord/K}(\p_r)$ is exactly the decomposition group of $\p_r$. Then by the uniqueness of ring class field, $H_R$ is the subfield of $H_\ord$ which is fixed by the decomposition group of $\p_r$ in $H_\ord/K$.
\end{proof}

In the case of the above theorem, it amounts to the fact that $H_R$ is the maximal subfield of $H_\ord$ in which $\p_r$ splits completely. Since Lemma \ref{lemma 3.2} holds for general $T$, then following the same argument we obtain the more general

\begin{corollary}
Let all notation be as above except removing the assumption of $T$. Let $T$ be a finite set of prime ideals of $K$, say $\set{\p_1,\cdots,\p_t}$. Then $H_R$ is the maximal subfield of $H_\ord$ where all the $\p_i$ split completely.
\end{corollary}

When $K$ is an imaginary quadratic field, thanks to the complex multiplication (CM) theory \cite{cox}, the ring class field $H_\ord$ defined by an order $\ord$ of $K$ can be generated by the singular moduli $j(\ord)$, i.e. $H_\ord=K(j(\ord))$. In this case, by using Theorem \ref{charact. thm}, together with one of the main theorem of CM \cite[Theorem 11.36]{cox} which says the Galois action of Artin symbol $\varphi_{H_\ord/K}^{\f_\ord}(\p)$ on $j(\ord)$ is given by $j(\p^{-1})$, we can obtain

\begin{corollary}
Let $K$ be an imaginary quadratic field, $R$ be a number ring in $K$ and $\ord_{K,T}$ be its integral closure in $K$ where $T=\set{\p_1,\cdots,\p_t}$. Set $H_\ord, H_R$ to be the ring class fields of $\ord, R$ respectively, and $f_i$ to be the residue class degree of $\p_i$ in $H_\ord$.  Then $H_R=\bigcap_{i=1}^t K\Big(\sum_{k=1}^{f_i}j(\p_i^{-k})\Big)$.
\end{corollary}

\section{Idelic Correspondence}\label{idelic-cor}

For a number field $K$, denote by $\mathbb{I}_K$ (resp. $\mathbb{A}_K$) the idele group (resp. adele ring) of $K$, $K^{ab}$ the maximal abelian extension of $K$. $C_K:=\mathbb I_K/K^*$ is the idele class group of $K$. The idele-theoretic formulation of class field theory states that, there is a unique continuous homomorphism, called the global norm residue symbol,
$$\psi_K: C_K \ra \Gal(K^{ab}/K)$$
such that the correspondence $U\mapsto \psi_K^{-1}(U)$ gives a bijection between open subgroups of $\Gal(K^{ab}/K)$ and open (hence of finite index) subgroups of $C_K$; moreover, for every finite abelian extension $L/K$, $\psi_K$ induces the isomorphism $C_K/N_{L/K}C_L\cong \Gal(L/K)$.

In this section, we give an explicit description of the open subgroup of $C_K$ corresponding to the number ring class field we defined in Section $\S$\ref{theory}.

Let $\Omega_K$ be the set of all places in $K$, $\infty_K$ the set of all infinite places in $K$.
Let $K_\p$ be the completion of $K$ at $\p$, $\oo_{K_\p}$ (or $\zz_p$ when $K=\Q$) be the valuation ring of $K_\p$ and $\m_\p$ be the maximal ideal of $\oo_{K_\p}$ for every $\p\in \Omega_K\setminus\infty_K$. We also write $\oo_{K_\p}=K_\p$ for $\p\in \infty_K$. Denote by $U_\p^i$ the multiplicative group $1+\m_\p^i$ for $i\geq 1$ and $U^0_\p=\oo_{K_\p}^*$. We write $A_{(\mathfrak{q})}$ for the localization of the ring $A$ at a prime ideal $\mathfrak{q}$ of $A$.

Let all notation be as in Section $\S$\ref{theory} and Section $\S$\ref{characterization}, e.g. $R$, $\so=\oo_{K,T}$, $\f$. Let $\ord=R\cap\ok$ and $\f_\ord$ be its conductor.

For proving the following Theorem \ref{idele-rcf} we need the lemma:
\begin{lemma}\label{ideletomodulus}\cite[Lemma 3.4]{Cohen_ccft}
For every modulus $\m=\prod_{\p\in\Omega_K}\p^{n_\p}$ of $K$, let $W_\m=\prod_{\p} U_\p^{n_\p}$. If $\p$ is real, set $U^1_\p=\mathbb{R}^*_{>0}$ . There is an isomorphism
\begin{equation*}
 \iota: \mathbb{I}_K/K^*W_\m\xrightarrow{\sim}Cl_\m
\end{equation*}
that maps $(x_\p)_\p$ to the class of $\prod_{\p\in\Omega\setminus\infty_K}\p^{\emph{ord}_\p(yx_\p)}$. Here $y\in K^*$ is a global element satisfying $yx_\p\in U_\p^{n_\p}$ for all $\p\mid\m$.
\end{lemma}

\begin{theorem}\label{idele-rcf}
For a number ring $R$ in $K$, let $H_R$ be the number ring class field defined by $R$. Let $f_\p=\emph{ord}_\p(\f_\ord)$ and $R_p=R\otimes_\zz \zz_p$. Then $\prod_{p\in\Omega_\Q}R_p^*$ is the open subgroup of $\mathbb I_K$ corresponding to the finite abelian extension $H_R/K$. That is, $\mathbb{I}_K/K^*\prod_pR_p^*\cong \Gal(H_R/K)$. Moreover,
\begin{equation}\label{idelic}
  \prod_{p\in\Omega_\Q}R_p^*=\prod_{\p\notin T\cup\infty_K}\Big(\ord/(\ord\cap\p^{f_\p})\Big)^*U_\p^{f_\p}\times\prod_{\p\in T\cup\infty_K}K_\p^*,
\end{equation}
where $\big(\ord/(\ord\cap\p^{f_\p})\big)^*U_\p^{f_\p}=\bigcup_{a\in(\ord/\ord\cap\p^{f_\p})^*}aU_\p^{f_\p}$ is an open subgroup of $\oo_{K_\p}^*$.

\end{theorem}

\begin{proof}
For any rational prime $p$, there is a natural $\Q_p$-algebra isomorphism \cite{cassel-frolich}
$$K\otimes_\Q\Q_p\xrightarrow[\simeq]{\varphi}\prod_{\p\mid p\ok}K_{\p},$$
which is also a homeomorphism. For any $\p\mid p\ok$, we compose it with the projection to $K_\p$, and denote it by $\varphi_\p: K\otimes_\Q\Q_p\twoheadrightarrow K_\p$. Note that it is a closed (and open) map. And the composition $(\mbox{id}\times 1)\comp\varphi_\p:K\hookrightarrow K\otimes_\Q\Q_p\twoheadrightarrow K_\p$ is an embedding of $K$ to $K_\p$. We denote it by $i_\p$. It is also well-known that $\varphi_\p(\ok\otimes_\zz\zz_p)=\oo_{K_\p}$, since the image $\varphi_\p(\ok\otimes_\zz\zz_p)\subseteq \oo_{K_\p}$ is closed in $K_\p$ and contains $\ok$ which is dense in $\oo_{K_\p}$.

First we prove this theorem in order case, i.e. $R=\ord$. In this case $T=\emptyset$ and we are reduced to the following
\begin{equation}\label{idel.ord.case}
   \prod_{p\in\Omega_\Q}\ord_p^*=\prod_{\p\in\Omega_K\setminus\infty_K}\Big(\ord/(\ord\cap\p^{f_\p})\Big)^*U_\p^{f_\p}\times\prod_{\p\in \infty_K}K_\p^*.
\end{equation}
For $\p\in\infty_K$, since $K=\Q\cdot\ord$ then $(\ord\otimes_\zz\mathbb{R})^*=(K\otimes_\Q\mathbb{R})^*=\prod_{\p\in \infty_K}K_\p^*$. For $\p\in\Omega_K\setminus\infty_K$, assume $\p\cap\zz=p\zz$. First we claim that $C_\p:=\varphi_\p(\ord\otimes_\zz\zz_p)$ is the closure $\bar\ord$ of $\ord$ in $K_\p$ (under the embedding $i_\p$ which we will omit if there is no confusion).

Since $\ord\otimes_\zz\zz_p$ is closed in $K\otimes_\Q\Q_p$ so $C_\p$ is closed in $K_\p$, which implies $C_\p\supseteq\bar\ord$. On the other hand, for any two elements $a_1, a_2\in\bar\ord$, and for any integer $l$, there exist positive integers $n_1, n_2$ such that $a_1a_2+\m_\p^l$ contains $(a_1+\m_\p^{n_1})(a_2+\m_\p^{n_2})$. Since $(a+\m_\p^{n_1})\cap\ord\neq\emptyset$ for $i=1,2$, it follows that $(a_1a_2+\m_\p^l)\cap\ord\neq\emptyset$, which means that $a_1a_2\in\bar\ord$. It is clear that $\bar\ord$ contains $\zz_p$ and $\ord$. Then $\bar\ord$ contains $\zz_p\cdot\ord=C_\p$. Therefore $C_\p=\bar\ord$ in $K_\p$.

Since $\f_\ord\ok\subseteq\ord$, then $\varphi_\p(\f\ok\otimes_\zz\zz_\p)\subseteq\varphi_\p(\ord\otimes_\zz\zz_p)=C_\p$ where the L.H.S. is equal to $\f\oo_{K_\p}=\m_\p^{f_\p}$ in $K_\p$. Note that when $f_\p=0$ it means that $C_\p=\oo_{K_\p}$. Hence $U^{f_\p}_\p$ is contained in $C_\p^*$. Then we have the following commutative diagram with exact rows:
\begin{equation*}
\xymatrix @-0.5pc {
&1\ar[r] & U^{f_\p}_\p \ar[r]\ar@{=}[d] & C_\p^*=\bar\ord^* \ar[r]\ar@{_{(}->}[d] & (\bar\ord/\m_\p^{f_\p})^* \ar[r]\ar@{_{(}->}[d] &1\\
&1\ar[r] & U^{f_\p}_\p \ar[r] & \oo_{K_\p}^* \ar[r] & (\oo_{K_\p}/\m_\p^{f_\p})^* \ar[r] &1
}
\end{equation*}
It is well-known that $\oo_{K_\p}/\m_\p^{f_\p}\cong\ok/\p^{f_\p}$ following \cite[Proposition 2.7]{janusz}. For orders, there is a natural monomorphism $\ord/\ord\cap\p^{f_\p}\hookrightarrow\bar\ord/\m_\p^{f_\p}$ since $\m_\p^{f_\p}\cap\ord=(\m_\p^{f_\p}\cap\ok)\cap\ord=\p^{f_\p}\cap\ord$. Here we should notice that it may happen $\p^{f_\p}\cap\ord\neq(\p\cap\ord)^{f_\p}$ when $\p\mid\f_\ord$. And for any $x\in\bar\ord$, there must be some $a\in\ord$ such that $x\equiv a\mbox{ mod }\m_\p^{f_\p}$ since $\ord$ is dense in $\bar\ord$. Therefore  $\ord/\ord\cap\p^{f_\p}\cong\bar\ord/\m_\p^{f_\p}$ and we obtain
$$C_\p^*=\bigcup_{a\in(\ord/\ord\cap\p^{f_\p})^*}aU_\p^{f_\p}.$$
Together with $\ord_p^*=(\ord\otimes_\zz\zz_p)^*\cong\prod_{\p\mid p\zz}C_\p^*$ then we obtain (\ref{idel.ord.case}). Note that $C^*_\p=\oo_{K_\p}^*$ for all $\p\nmid\f_\ord$.

It is now obvious that $\prod_p\ord_p^*$ is an open subgroup of $\mathbb{I}_K$ and contains $W_{\f_\ord}$ described in Lemma \ref{ideletomodulus}. Next we shall show that
\begin{equation}\label{idelePic}
  \mathbb{I}_K/K^*\prod_{p\in\Omega_\Q}\ord_p^*\cong\Pic(\ord).
\end{equation}
Recall that $\Pic(\ord)\cong J^{\f_\ord}_K/P^{\f_\ord}_{K,\ord}$ where $P^{\f_\ord}_{K,\ord}=\set{(\alpha_1/\alpha_2)\ok|\alpha_i\in\ord, \alpha_i\ord+\f_\ord=\ord\mbox{ for }i=1,2}$. Replace the modulus $\m$ in Lemma \ref{ideletomodulus} by $\f_\ord$; for proving (\ref{idelePic}) it remains to verify that $\iota(K^*\prod_p\ord_p^*)=P_{K,\ord}^{\f_\ord}$.
First we should notice the fact: by Chinese Remainder Theorem $(\ok/\f_\ord)^*\cong\prod_{\p\mid\f_\ord}(\ok/\p^{f_\p})^*$. Under this isomorphism we have $(\ord/\f_\ord)^*\cong\prod_{\p\mid\f_\ord}(\ord/\ord\cap\p^{f_\p})^*$.

It is clear that $\iota(K^*)\subseteq P^{\f_\ord}_{K,1}\subseteq P^{\f_\ord}_{K,\ord}$. Take any $(x_\p)_\p\in\prod_p\ord_p^*$. Following (\ref{idel.ord.case}) for every $\p\mid\f_\ord$ there exists an $a_\p\in(\ord/\ord\cap\p^{f_\p})^*$ such that $x_\p\equiv a_\p\mbox{ mod }\m_\p^{f_\p}$, or equivalently $x_\p a_\p^{-1}\in U_p^{f_\p}$ for some $a_\p^{-1}\in(\ord/\ord\cap\p^{f_\p})^*$. Then there exists a $y\in(\ord/\f_\ord)^*$ i.e $y\in\ord$ and $y\ord+\f_\ord=\ord$, such that $y\equiv a_\p^{-1}\mbox{ mod }(\ord\cap\p^{f_\p})$ for all $\p\mid\f_\ord$, which implies that $yx_\p\in U_\f^{f_\p}$ for all $\p\mid\f_\ord$. Hence $\iota((x_\p)_\p)=\prod_{\p\in\Omega_K\setminus\infty_K}\p^{\text{ord}_\p(yx_\p)}=y\ok\in P^{\f_\ord}_{K,\ord}$ since $\text{ord}_\p(x_\p)=0$ for all finite $\p$, which shows that $\iota(K^*\prod_p\ord_p^*)\subseteq P_{K,\ord}^{\f_\ord}$. If we take an idele $(x_\p)_\p$ such that its image of $\iota$ belongs to $P^{\f_\ord}_{K,\ord}$, it means there is a $y\in K^*$ and two $\alpha_1,\alpha_2\in\ord$ with $\alpha_i\ord+\f=\ord$, such that $yx_\p\in U_\p^{f_\p}$ for all $\p\mid\f_\ord$ and $\text{ord}_\p(yx_\p)=\text{ord}_\p(\alpha_1/\alpha_2)$ for all finite $\p$, of which the latter implies that $\text{ord}_\p((\alpha_2/\alpha_1)yx_\p)=0$ for all finite $\p$. Hence $(\alpha_2/\alpha_1)yx_\p\in U^0_\p=\oo_{K_\p}^*$ for all finite $\p$. Moreover, for $\p\mid\f_\ord$, since $yx_p\in U_\p^{f_\p}$ and the residue class of $\alpha_2/\alpha_1$ belongs to $(\ord/\ord\cap\p^{f_\p})^*$ so $(\alpha_2/\alpha_1)yx_\p\in C_\p^*$. Therefore $(\alpha_2y/\alpha_1)(x_\p)_\p\in\prod_p\ord_p^*$ which shows that $\iota(K^*\prod_p\ord_p^*)\supseteq P_{K,\ord}^{\f_\ord}$. Hence we obtain (\ref{idelePic}), and following the idele-theoretic class field theory we know that the open subgroup $\prod_{p\in\Omega_\Q}\ord_p^*$ of $\mathbb{I}_K$ corresponds to the ring class field $H_\ord$ of $\ord$.

Next we prove this theorem in the general number ring case. For $\p\in\infty_K$, as in order case $(R\otimes_\zz\mathbb{R})^*=\prod_{\p\in\infty_K}K_\p^*$. For $\p\in\Omega_K\setminus\infty_K$, assume $\p\cap\zz=p\zz$. We use the same symbol $C_\p$ to denote $\varphi_\p(R\otimes_\zz\zz_p)$ as in the order case. When $\p\in T$, we claim that $C_\p^*=K_\p^*$. Recall that $\ord=R\cap\ok$ and $\p$ is prime to $\f_\ord$, and it has been proved in order case that $\varphi_\p(\ord\otimes_\zz\zz_p)=\oo_{K_\p}$. Then $\oo_{K_\p}$ is contained in $C_\p$. On the other hand, there exists an $\epsilon\in R\subseteq\ C_\p$ such that $\text{ord}_\p(\epsilon)<0$. Hence $C_\p=K_\p$. When $\p\notin T$, we claim that $C_\p=\varphi_\p(\ord\otimes_\zz\zz_p)$.

Recall that $\varphi_\p(\ord\otimes_\zz\zz_p)$ is exactly the closure $\bar\ord$ of $\ord$ in $K_\p$. Denote by $\tilde\p$ the prime ideal $\p\cap\ord$ of $\ord$.  We claim that the localization $\ord_{(\tilde\p)}$ of $\ord$ at $\tilde\p$ is contained in $\bar\ord$. Take any $1/s\in\ord_{(\tilde\p)}$ where $s\in\ord\setminus\tilde\p$. For any (positive) integer $l$, since $s\ord+\tilde\p^l=\ord$ (the maximal ideal in order $\ord$ that contains $\tilde\p^l$ is $\tilde\p$), so there exists an $a\in\ord$ such that $s\cdot a\equiv 1\text{ mod }\m_\p^l$. Then $1/s\equiv a\text{ mod }\m_\p^l$ since $s\notin \m_\p$, which implies $(1/s+\m_\p^l)\cap\ord\neq\emptyset$. Therefore $\ord_{(\tilde\p)}\subseteq\varphi_\p(\ord\otimes_\zz\zz_p)$. By Proposition \ref{local.ord.} in next section it follows that $R=\ord[\frac{1}{\epsilon}]$ for some $T$-unit $\epsilon\in\ord$. And $\epsilon\notin\p$ since $\p\notin T$. Then $\epsilon\in\ord\setminus\tilde\p$. It implies that $\frac{1}{\epsilon}\in\ord_{(\tilde\p)}\subseteq\varphi_\p(\ord\otimes_\zz\zz_p)$. Hence $\varphi_\p(R\otimes_\zz\zz_p)\subseteq\varphi_\p(\ord\otimes_\zz\zz_p)$ and then $C_\p=\varphi_\p(\ord\otimes_\zz\zz_p)$.

As a result, together with (\ref{idel.ord.case}) we can obtain (\ref{idelic}). Furthermore, as for any $\p\in T$, $\iota^{-1}(\p)=(1)_{\mathfrak{q}\neq\p}\times \tau_\p\in\mathbb{I}_K$, denoted by $\theta_\p$, where $\tau_\p\in K_\p^*$ is an uniformizer for $\p$. It follows that $\prod_{p\in\Omega_\Q}R^*_p=\prod_{p\in\Omega_\Q}\ord_p^*\cdot\langle\theta_\p\rangle_{\p\in T}$. By Lemma \ref{lemma 3.2} we know that $\Pic(R)\cong\Pic(\ord)/\prod_{\p\in T}\p^\zz$.  Therefore, the $\iota$ in Lemma \ref{ideletomodulus} induces an isomorphism
\begin{equation*}
  \mathbb{I}_K/K^*\prod_{p\in\Omega_\Q}R^*_p\cong\Pic(R).
\end{equation*}
Following the idele-theoretic class field theory, the open subgroup $\prod_{p\in\Omega_\Q}R_p^*$ of $\mathbb{I}_K$ corresponds to the number ring class field $H_R$ of $R$. This completes our proof.\end{proof}

\section{An Application to Norm Form Equations}\label{application}

Before studying the application, we investigate further the structure of number rings in $\Q$ and in $K$.

The smallest number ring in $\Q$ is $\zz$. When talk about a number ring $R$ in $\Q$, we mean it to be a finitely generated $\zz$-algebra in $\Q$ (under the \textbf{Assumption} in Section $\S$\ref{basic-notation-and-results}), say $\zz[\frac{b_1}{a_1},\cdots,\frac{b_t}{a_t}]$ where $a_i$, $b_i$ are coprime integers. First, $\zz[\frac{b_i}{a_i}]=\zz[\frac{1}{a_i}]$ since $a_i$ and $b_i$ are coprime; second, $\zz[\frac{1}{a_i},\frac{1}{a_j}]=\zz[\frac{1}{a_i a_j}]$. Hence $R$ has the form $\zz[\frac{1}{a}]$ where $a$ is a square-free positive integer, which is a localization of $\zz$ by the set of the powers of $a$. Note that it is still a principal ideal domain (PID).

For a number ring $R$ in $K$, its integral closure $\so$ in $K$ is the ring of $T$-integers $\oo_{K,T}$ for some $T$. If set $\ord=R\cap \ok$, then $T$ is finite if and only if $\so=\ok[\frac{1}{\epsilon}]$ for some $T$-unit $\epsilon\in\ok$ if and only if $R=\ord[\frac{1}{\epsilon'}]$ for some $T$-unit $\epsilon'\in\ord$. Since $\oo_{K,T}$ can be generated by $T$-units as $\ok$-algebra, and the rank of the group of $T$-units is $(r+\#T)$ where $r$ is the rank of $\ok^*$ \cite[Theorem 10.9]{arith_nr}, then $T$ is finite if and only if $\so=\ok[\frac{1}{u_1},\cdots,\frac{1}{u_t}]=\ok[\frac{1}{\epsilon}]$ where $\epsilon=\prod u_i$ is an integral $T$-unit. Moreover, if $\so=\ok[\frac{1}{\epsilon}]$ for a $T$-unit $\epsilon\in \ok$, then $\frac{1}{\epsilon^m}$ belongs to $R$ for some positive integer $m$ since $[\so^*: R^*]$ is finite, which implies $\ord[\frac{1}{\epsilon^m}]\subseteq R$. On the other hand, if $x\in R$, then there is a positive integer $e$ such that $x\epsilon^{me}\in\ok$. Then $x\epsilon^{me}\in \ok\cap R=\ord$ which implies $R\subseteq\ord[\frac{1}{\epsilon^m}]$, and hence $R=\ord[\frac{1}{\epsilon^m}]$. For the sufficiency of the second part, it is clear since $\so=R\ok$. Through all these argument now we obtain

\begin{proposition}\label{local.ord.}
For a number ring $R$ in $K$, suppose its integral closure $\so=\oo_{K,T}$ for some $T$. Let $\ord=\ok\cap R$. Let $\f$ be the conductor of $R$, $\f_\ord$ be that of $\ord$. Under the \textbf{Assumption} in Section $\S$\ref{basic-notation-and-results}, then there exist some $T$-unit $\epsilon$ belonging to $\ord$ such that $R$ is a localization of $\ord$ by the set of the powers of $\epsilon$. Hence $\oo_{K,T}$ is a localization of $\ok$ by the set of powers of $\epsilon$. Moreover, $\f_\ord$ is equal to $\f\cap\ok$, and $\f$ is equal to $\f_\ord\oo_{K,T}$.
\end{proposition}

\begin{proof}
 It has been proved in Lemma \ref{lemma 3.2} that $\f_\ord\supseteq\f\cap\ok$. And it is also obvious that $\f_\ord\ok[\frac{1}{\epsilon}]\subseteq\ord[\frac{1}{\epsilon}]$, which implies that $\f_\ord\subseteq\f\cap\ok$ and hence $\f_\ord=\f\cap\ok$. Since $\f$ is an $\oo_{K,T}$-ideal so $\f_\ord\oo_{K,T}$ is contained in $\f$. For any $x\in\f$, it satisfies $x\ok[\frac{1}{\epsilon}]\subseteq\ord[\frac{1}{\epsilon}]$. Then there exists a positive integer $l$ such that $\epsilon^lx\ok\subseteq\ord$. It follows that $\epsilon^l x\in\f_\ord$. Then $x\in\f_\ord[\frac{1}{\epsilon}]=\f_\ord\oo_{K,T}$. The proof is complete.
\end{proof}

When extending a number ring in $\Q$ to that in $K$, we have

\begin{corollary}\label{complete n.r.}\label{fin.gen.ZTmod}
Given a number ring $\zz[\frac{1}{a}]$ in $\Q$ and a number ring $R$ in $K$. Set $T$ to be the finite set of all the prime ideals of $\ok$ dividing $a\ok$. Then
\begin{itemize}
  \item[(1)] The integral closure of $\zz[\frac{1}{a}]$ in $K$ is $\ok[\frac{1}{a}]$, which is equal to $\oo_{K,T}$. In this case, $\oo_{K,T}$ is a free $\zz[\frac{1}{a}]$-module of rank $n$.
  \item[(2)] If $R$ is a finitely generated $\zz[\frac{1}{a}]$-module, then $R=\ord[\frac{1}{a}]$, where $\ord=\ok\cap R$.
\end{itemize}
\end{corollary}

\begin{proof}
The integral closure of $\zz[\frac{1}{a}]$ in $K$ must contain $\ok[\frac{1}{a}]$, which is a number ring in $K$. And their integral closures in $K$ are identical. It follows from Proposition \ref{basic prop. of number ring} (3) that the integral closure of $\ok[\frac{1}{a}]$ in $K$ is itself. Hence the integral closure of $\zz[\frac{1}{a}]$ in $K$ is $\ok[\frac{1}{a}]$, and is clearly equal to $\oo_{K,T}$. Since $\oo_{K,T}$ is torsion-free and $\zz[\frac{1}{a}]$ is a PID, so it is a free $\zz[\frac{1}{a}]$-module with rank equal to the degree of $K/\Q$.

If $R$ is a finitely generated  $\zz[\frac{1}{a}]$-module, then $R$ is contained in the integral closure of $\zz[\frac{1}{a}]$ in $K$, which is $\ok[\frac{1}{a}]$. Then for any element $x\in R$, there exists a positive integer $m$ such that $a^m x\in \ok\cap R=\ord$, which implies that $x\in\ord[\frac{1}{a}]$. The proof is complete.
\end{proof}

Rings class fields are originally applied to the integral solvability of quadratic equations of the form $p=x^2+ny^2$ over $\zz$, see \cite{cox}. Then Wei and Xu \cite{multi_norm} construct a class of idele groups called $\mathbf{X}$-\emph{admissible subgroups} for determining the integral points for multi-norm tori, and interpret the $\mathbf{X}$-admissible subgroup in terms of Brauer-Manin obstruction. By this method they give an idelic form criterion of the integral solvability of a quadratic equation of the form $a=x^2\pm dy^2$ over $\zz$ for an arbitrary integer $a$. It can recover the classical result in \cite{cox} when $a$ is prime and the $\mathbf{X}$-admissible subgroup in this case is exactly the idelic correspondence of the ring class field of the order $\zz[\sqrt{-d}]$ in the quadratic field $\Q(\sqrt{- d})$. Also they consider general norm form equations for higher dimensional tori, but do not give a certain description of the class field corresponding to the $\mathbf{X}$-admissible subgroup. In \cite{Integralrep.}, Lv, Shentu and Deng apply this method to the case of relative quadratic extension, say $F(\sqrt{-d})/F$, and consider the $\ord_F$-integral solvability of a quadratic equation $ax^2+bxy+cy^2+g=0$ over $\oo_F$. As in the case of absolute quadratic extension (i.e. $\Q(\sqrt{-d})/\Q$), under some assumptions one can choose the $\mathbf{X}$-admissible subgroup to be the ring class field of the order $\oo_F[\sqrt{-d}]$ in $F(\sqrt{-d})$.

Now we may extend these results to the relative extension case for a higher degree norm form equation over a number ring.

Let $E/F$ be a finite extension of number fields of relative degree $n$, $R$ be a number ring in $F$. Let $W$ be a number ring in $E$ which is an $R$-module that can be generated by $n$ elements $\{\alpha_1,\cdots,\alpha_n\}$ and $f(x_1,\cdots,x_n)=N_{E/F}(x_1\alpha_1+\cdots+x_n\alpha_n)\in R[x_1,\cdots,x_n]$. Note that the second condition is automatic if $R$ is integral closed in $F$. For any $\p\in\Omega_F$, denote by $\bar R_\p$ the closure of $R$ in the complete field $F_\p$, $W_\p=W\otimes_R\bar R_\p$ and $E_\p=E\otimes_F F_\p$.

 For any $b\in R$, let $\mathbf{X}$ be the affine scheme over $R$ defined by the equation $f(x_1,\cdots,x_n)=b$. Obviously $f(x_1,\cdots,x_n)=b$ has a solution belonging to $R^n$ ($R$-integral solution) if and only if $\mathbf{X}(R)\neq\emptyset$ ($R$-point). Next we follow some of the notation used in \cite{Integralrep.}.

Denote by $R_{E/F}(\mathbb{G}_{m,E})$ the Weil restriction of $\mathbb{G}_{m,E}$ to $F$. Denote by $T$ the kernel of the norm map $N: R_{E/F}(\mathbb{G}_{m,E})\ra \mathbb{G}_{m,F}$. It is a torus and $\mathbf{T}:=\text{Spec}(R[x_1,\cdots,x_n]/(f-1))$ is an $R$-model of $T$. Let $X_F$ be the generic fiber of $\mathbf{X}$. Then it is natural that $X_F$ is a $T$-torsor. Hence given any $P\in X_F(F)\subseteq E^*$ (assume that $X_F(F)\neq\emptyset$), there is an isomorphism $\phi_P: X_F\cong T, x\mapsto P^{-1}x$. Particularly $\phi_P(\mathbb{A}_F): X_F(\mathbb{A}_F)\cong T(\mathbb{A}_F)$.

 Denote by $\lambda_E$ the natural injective homomorphism $T(\mathbb{A}_F)\hookrightarrow\mathbb{I}_E$. Note that $\bar R_\p=\oo_{F_\p}$ for almost all $\p\in\Omega_F$. Then $\prod_{\p\in\Omega_F}\mathbf{T}(\bar R_\p)$ (resp. $\prod_{\p\in\Omega_F}\mathbf{X}(\bar R_\p)$) can be viewed as a subset of $T(\mathbb{A}_F)$ (resp. $X_F(\mathbb{A}_F)$). Moreover, for any $p\in\Omega_\Q$,
 $$W_p=W\otimes_\zz\zz_p=W\otimes_R(R\otimes_\zz\zz_p)=W\otimes_R\prod_{\p\mid p\oo_F}\bar R_\p=\prod_{\p\mid p\oo_F}W_\p.$$
 Hence $\prod_{p\in\Omega_\Q}W_p^*=\prod_{\p\in\Omega_F}W_\p^*$, which, by Theorem \ref{idele-rcf}, is an open subgroup of $\mathbb{I}_E$ and corresponds to the ring class field of $W$.

 The monomorphism $\lambda_E$ maps the subsets $T(F)$ into $E^*$ and $\prod_{\p\in\Omega_F}\mathbf{T}(\bar R_\p)$ into $\prod_{\p\in\Omega_F}W_\p^*$. Hence it induces a homomorphism
$$\tilde\lambda_E: T(\mathbb{A}_F)/T(F)\prod_{\p\in\Omega_F}\mathbf{T}(\bar R_\p)\ra \mathbb{I}_E/E^*\prod_{\p\in\Omega_F}W_\p^*.$$
The injectivity of $\tilde\lambda_E$ is crucial. Following the definition from \cite{multi_norm} an open subgroup $\Xi$ of $\mathbb{I_K}$ is called $\mathbf{X}$-admissible if it contains $\lambda_E(\prod_{\p\in\Omega_F}\mathbf{T}(\bar R_\p))$ and the induced $\tilde\lambda_E$ is injective.

On the other hand, through the isomorphism $\phi_P$, we can read $\prod_{\p\in\Omega_F}\mathbf{X}(\bar R_\p)$ as a subset of $\mathbb{I}_E$ by the following
$$\prod_{\p\in\Omega_F}\mathbf{X}(\bar R_\p)\hookrightarrow X_F(\mathbb{A}_F)\xrightarrow[\phi_P]{\cong} T(\mathbb{A}_F)\hookrightarrow\mathbb{I}_E\xrightarrow{\cdot P}\mathbb{I}_E.$$
Denote this composite map by $\tilde f_E$. Actually this embedding $\tilde f_E$ can be naturally interpreted by another way, not requiring that $X_F(F)\neq\emptyset$: for every place $\p$ it is clear that
$$\mathbf{X}(\bar R_\p)=\{x\in W_\p: \text{N}_{E_\p/F_\p}(x)=b\}$$
through the bijection between $(x_1,\cdots,x_n)\in \bar R_\p^n$ and $x=\sum x_i\alpha_i\in W_\p$. Since $W_\p=\prod_{\mathfrak{P}\mid\p}\oo_{E_\mathfrak{P}}$ for almost all $\p\in\Omega_F$ following Theorem \ref{idele-rcf}, and if $\p\nmid a\oo_F$ then $\text{N}_{E_\p/F_\p}(x)=b$ implies that $x\in W^*_\p$, then $\mathbf{X}(\bar R_\p)\subseteq \prod_{\mathfrak{P}\mid\p}\oo^*_{E_\mathfrak{P}}$ for almost all $\p\in\Omega_F$.

Denote by $H_W$ the ring class field of $W$. Let $\psi_{H_W/E}:\mathbb{I}_E/E^*\prod_{\p\in\Omega_F} W^*_\p\ra \Gal(H_W/E)$ be the induced isomorphism under the global norm residue symbol $\psi_E$. Under some assumptions, $\prod_{\p\in\Omega_F}W^*_\p$ is an admissible subgroup for $\mathbf{X}$, and hence gives a criterion of the solvability of $f(x_1,\cdots,x_n)=a$ over $R$.

\begin{proposition}\label{prop_integralpt}
Let $\mathcal{U}$ be a complete set of representatives of $R^*/(R^*)^n$. Assume that for every $u\in\mathcal{U}$, if the norm form equation $f(x_1,\cdots,x_n)=u$ is solvable over $\bar R_\p$ for all $\p\in\Omega_F$ then it is solvable over $R$. Then $\mathbf{X}(R)\neq\emptyset$ if and only if $X_F(F)\neq\emptyset$ and there exists an $x\in\prod_{\p\in\Omega_F}\mathbf{X}(\bar R_\p)$ such that $\psi_{H_W/E}(\tilde f_E(x))=1$.
\end{proposition}

\begin{proof}
It is analogous to the case when $R=\oo_F$. See the proof of Lemma 1 and Proposition 1 in \cite{Integralrep.}. We should notice that $F^*\cap\prod_{\p\in\Omega_F}\bar R^*_\p=\cap_{\p\in\Omega_F}(\bar R^*_\p\cap F^*)=\cap_{\p\in\Omega_F\setminus\infty_F}R^*_{\p\cap R}=R^*$, and it follows that $\mathbf{X}(F)\cap\prod_{\p\in\Omega_F}\mathbf{X}(\bar R_\p)=\mathbf{X}(R)$.
\end{proof}

In the following we consider the absolute case $F=\Q$. When $E/\Q$ is an imaginary quadratic extension and $R=\zz$, the classic ideal-version criterion in \cite{cox} can be recovered by the above idele-version criterion, see Corollary 4.2 in \cite{multi_norm}. In the case of real quadratic extension, the problem of the sign appeared in the norm form equations can not been dealt with effectively by the ideal-version method (also known as Gauss's method). But some connection can be built between these two version criterions. More generally we have the following

\begin{proposition}
Let notation be as above. Take $F=\Q$, $R=\zz$. Then $W$ is contained in $\oo_E$ hence an order in $E$. Let $\mathbf{X}'$ be the affine scheme over $\zz$ defined by the equation $f(x_1,\cdots,x_n)=-b$. Take $b=\ell$ to be a rational prime and $\ell\nmid[\oo_E:W]$.
 Then the following are equivalent:
 \begin{itemize}
   \item[(1)] $\ell=f(x_1,\cdots,x_n)$ or $-\ell=f(x_1,\cdots,x_n)$ is solvable over $\zz$;
   \item[(2)] There exists an $x\in\prod_{p\in\Omega_\Q}\mathbf{X}(\zz_p)$ or $x\in\prod_p\mathbf{X}'(\zz_p)$ such that $\psi_{H_W/E}(\tilde f_E(x))=1$;
   \item[(3)] There exists a prime ideal $\p_\ell$ of $E$ lying over $\ell$ such that $N_{E/\Q}(\p_\ell)=\ell$ and $\p_\ell$ splits completely in $H_W$.
 \end{itemize}
\end{proposition}

\begin{proof}
1) $\Rightarrow$ 2) is trivial.

For 2) $\Rightarrow$ 3), there exists an $x\in\prod_{p\in\Omega_\Q}\mathbf{X}(\zz_p)$ such that $\psi_{H_W/E}(\tilde f_E(x))=1$ if and only if
\begin{equation}\label{Artincondition}
  \tilde f_E(\prod_{p\in\Omega_\Q}\mathbf{X}(\zz_p))\cap E^*\prod_{p\in\Omega_\Q}W^*_p\neq\emptyset,
\end{equation}
if and only if there exists an $(x_p)\in\prod_p W_p$ with $N_{E_p/\Q_p}(x_p)=\ell$ for all $p\in\Omega_\Q$ belongs to $E^*\prod_{p\in\Omega_\Q}W^*_p$.

Recall from Theorem \ref{idele-rcf}, $W_p\cong\prod_{\p\mid p\oo_E}\overline W_\p$ where $\overline W_\p$ is the closure of $W$ in the complete field $E_\p$. Write $\varphi_\p: W_p\ra\overline W_\p$ for the composition of this isomorphism with the projection to the place $\p$. Then $N_{E_p/\Q_p}(x_p)=\prod_{\p\mid p}N_{E_\p/\Q_p}(\varphi_\p(x_p))=\ell$, and then $\text{ord}_p(\ell)=\sum_{\p\mid p}f(\p/p)\text{ord}_\p(\varphi_\p(x_p))$ where $f(\p/p)$ is the residue class degree of $\p$ over $p$.

When $p\neq\ell$, we have $\mathrm{ord}_\p(\varphi_\p(x_p))=0$ for all $\p\mid p\oo_E$. Hence $\varphi_\p(x_p)\in\oo_{E_\p}^*$ and then $\varphi_\p(x_p)\in\overline W^*_\p$ since it is easy to show that $\oo^*_{E_\p}\cap\overline W_\p=\overline W^*_\p$, and therefore $x_p\in W^*_p$. When $p=\ell$, there is exactly one prime ideal $\p_\ell$ lying over $\ell$ such that $\mathrm{ord}_{\p_\ell}(\varphi_{\p_\ell}(x_\ell))=1$ and $f(\p_\ell/\ell)=1$. For other $\p\neq\p_\ell$ lying over $\ell$, $\mathrm{ord}_\p(\varphi_\p(x_\ell))=0$ and hence $\varphi_\p(x_\ell)\in\overline W^*_\p$. Note that $\varphi_{\p_\ell}(x_\ell)$ is an uniformizer in $E_{\p_\ell}$, denoted by $\pi_{\p_\ell}$, and $N_{E/\Q}(\p_\ell)=\ell$. Therefore, by multiplying some element in $\prod_p W^*_p$, $(x_p)$ is equal to the idele $i_\ell=(\cdots,1,\pi_{\p_\ell},1,\cdots)$ with $\pi_{\p_\ell}$ at the place $\p_\ell$ and $1$ others.

Following Theorem \ref{idele-rcf}, since $\ell\nmid[\oo_E:W]$ we have $\iota(i_{\ell})=\p_\ell$. Hence if $(x_p)$ belongs to $E^*\prod_{p\in\Omega_\Q}W^*_p$ then equivalently $i_\ell$ belongs to $E^*\prod_{p\in\Omega_\Q}W^*_p$, and then $\iota(i_{\ell})=\p_\ell=1$ in $\Pic(W)$ which amounts to that $\p_\ell$ splits completely in $H_W$.

For 3) $\Rightarrow$ 1), note that $\ell\nmid[\oo_E:W]$. Hence (3) means that $\p_\ell=\alpha\oo_E$ for some $\alpha\in\ord$ and $N_{E/\Q}(\p_\ell)=|N_{E/\Q}(\alpha)|=\ell$, which clearly implies (1).
\end{proof}

\begin{remark}
When $n$ is odd, (1) is equivalent to that $\ell=f(x_1,\cdots,x_n)$ is solvable over $\zz$, and (2) to that there exists an $x\in\prod_{p\in\Omega_\Q}\mathbf{X}(\zz_p)$ such that $\psi_{H_W/E}(\tilde f_E(x))=1$.
\end{remark}

If we take $R=\zz[\frac{1}{a}]$ for some positive integer $a$, then by Corollary \ref{fin.gen.ZTmod} $W=\zz[\frac{1}{a}]\ord$ for the order $\ord=W\cap\oo_E$. Using the ideal-version method we have

%

\begin{corollary}\label{main appli}
Let notation be as above. Take $b=\ell$ to be a rational prime and $\ell\nmid a$, $\ell\nmid [\oo_E:\ord]$. Then $\pm \ell a^m=f(x_1,\cdots,x_n)$ is solvable over $\zz$ for some non-negative integer $m$ if and only if there exists some prime ideal $\p_\ell$ lying over $\ell$ such that $N_{E/\Q}(\p_\ell)=\ell$ and $\p_\ell$ splits completely in the ring class field $H_W$. In particular, if $E/\Q$ is Galois and $W$ is integral closed in $E$, then the last condition can be simplified to that $\ell$ splits completely in $H_W$.
\end{corollary}

\begin{proof}
The norm form equation $\pm \ell a^m=f(x_1,\cdots,x_n)$ is solvable over $\zz$ for some $m\in\zz_{\geq 0}$ if and only if
 \begin{equation}\label{ideal-appl}
  \ell a^{m'}=N_{E/\Q}(\beta\oo_E)\quad \text{for some}~ \beta\in W, m'\in\zz_{\geq 0}.
 \end{equation}
  Let $T=\set{\text{prime ideals } \p | \p\mid a\oo_E}$, then the integral closure of $W$ in $E$ is $\oo_{E,T}$ by Corollary \ref{complete n.r.}. If (\ref{ideal-appl}) is true, then the prime ideal decomposition of $\beta\oo_E$ has the form $\p_\ell\prod_{\p\in T}\p^e$ with $N_{E/\Q}(\p_\ell)=\ell$. It is equivalent to that $\beta\oo_{E,T}\cap\oo_E=\p_\ell$. Hence (\ref{ideal-appl}) holds if and only if
   \begin{equation}\label{ideal-appl-2}
     \ell=N_{E/\Q}(\beta\oo_{E,T}\cap\oo_E)\quad\text{for some }\beta\in W.
   \end{equation}
  By the definition of the congruence subgroup  $P^{\f\m_0}_{E,W}$ determined by $W$ (c.f. Section $\S$\ref{Sec-congruence-subgroup}) and the number ring class field $H_W$, and note that $\ell\nmid a$, $\ell\nmid [\oo_E:\ord]$, then (\ref{ideal-appl-2}) is equivalent to that there exists some prime ideal $\p_\ell$ lying over $\ell$ such that $N_{E/\Q}(\p_\ell)=\ell$ and $\p_\ell$ splits completely in $H_W$.

  If $E/\Q$ is Galois and $W$ is integral closed in $E$ i.e. $W=\oo_{E,T}$, then the admissible modulus $\f\m_0=\prod_{\p\in T}\p$ for $H_W$ is Galois invariant and the congruence subgroup $P^{\f\m_0}_{E,W}=P^{\m_0}_{E,\oo_{E,T}}$ for $H_W$ is also Galois invariant since $\oo_{E,T}$ is Galois invariant. Hence $H_W$ is Galois over $\Q$. The proof is complete.
  \end{proof}

\section{Computation}\label{computation}
In this section, for a number ring $R$ we describe how to compute its Picard group, and then to compute its ring class field using the computer algebra system Pari/GP \cite{pari}.

Assume $R$ is given by an order $\ord$ and a $T$-unit $\epsilon$ in $\ord$. Since $\ord$ is a free $\zz$-module of rank $n$, so it can be represented by a $n\times n$ matrix with integer coefficients, where columns represent a basis of $\ord$ under a given set of integral basis of $K$. And $\epsilon$ is represented by a column under this basis.

Recall \cite{Cohen_rcg} that a finitely generated abelian group $\mathcal{G}$ is represented by a pair $(G,D_G)$ where $G$ ia s finite set of generators of $G$ given as a row vector, and $D_G$ a matrix in Smith Normal Form (SNF) representing relations of $G$, such that $GX=0$ if and only if $X$ is a $\zz$-linear combination of the columns of $D_G$ (when $\mathcal{G}$ is a multiplicative group and $G=(g_i)$ and $X=(x_i)$, then $GX$ is an abbreviation of $\prod g_i^{x_i}$).

Given a short exact sequence of finitely generated abelian groups
$$1 \ra \mathcal{A}\xrightarrow{\psi} \mathcal{B}\xrightarrow{\phi} \mathcal{C} \ra 1.$$
Assume $\psi$ and $\phi$ can be computed efficiently. It is often asked how to compute a group and solve its discrete logarithm problem (DLP) from the remaining two groups and their DLP algorithms. Both three cases have been discussed clearly in \cite{Cohen_rcg,gtm193}. These three routines will be used repeatedly in the following computation.

\subsection{Computing the Picard Group of $R$}\label{computation1}
For the number ring $R$, its integral closure is the $T$-integer ring $\oo_{K,T}$ for some $T$. Let $\f$ be the conductor of $R$. We may use the following exact sequence from \cite[Theorem 6.7]{arith_nr}
\begin{equation}\label{exact seq.}
  \oo_{K,T}^*\xrightarrow{\varphi} (\oo_{K,T}/\f)^*/(R/\f)^* \xrightarrow{\psi} \Pic(R)\xrightarrow{\phi} \Cl(\oo_{K,T})\ra 1
\end{equation}
to compute the Picard group of $R$.


\subsubsection{Computing Other Terms in the Exact Sequence.}
For the first term in (\ref{exact seq.}), Pari/GP has the function {\ttfamily bnfsunit} to compute the group of $T$-units given any number field $K$ and finite set $T$ of prime ideals of $K$ as input. For the second term in (\ref{exact seq.}), note that $\oo_{K,T}$ is no longer a finitely generated $\zz$-module. We need the following

\begin{lemma}\label{kernelphi}
Let $R, \oo_{K,T}$ and $\f$ be as above. By Corollary \ref{fin.gen.ZTmod} we have $R=\zz[\frac{1}{a}]\ord$ where $\ord$ is the order $R\cap\ok$. Denote by $\f_\ord$ the conductor of $\ord$. Then we have $\oo_{K,T}/\f = \ok/\f_\ord$ and $R/\f= \ord/\f_\ord$.
\end{lemma}

\begin{proof}
By Proposition \ref{local.ord.} $\f_\ord=\f\cap\ok$, it follows that $\ok/\f_\ord\subseteq\oo_{K,T}/\f$. And there is some $T$-unit $\epsilon\in\ord$ such that $\oo_{K,T}=\ok[\frac{1}{\epsilon}]$ and $\f=\f_\ord[\frac{1}{\epsilon}]$. Then $\oo_{K,T}/\f=\ok/\f_\ord[\frac{1}{\bar\epsilon}]$ where we write $\bar\epsilon$ for the image of $\epsilon$ in $\ok/\f_\ord$. Since $\epsilon$ is a $T$-unit so $\bar\epsilon\in(\ok/\f_\ord)^*$, and hence $\ok/\f_\ord[\frac{1}{\bar\epsilon}]=\ok/\f_\ord$. Following the same argument we can prove $R/\f= \ord/\f_\ord$.
\end{proof}

By this Lemma we have $(\oo_{K,T}/\f)^*/(R/\f)^*=(\ok/\f_\ord)^*/(\ord/\f_\ord)^*$. In Pari/GP the function {\ttfamily idealstar} can be used to compute the multiplicative group $(\ok/\f_\ord)^*$ and {\ttfamily ideallog} to compute the discrete logarithm of an element in it. Complete algorithm is described in \cite{Cohen_rcg}. And \cite{residuering} made an improvement on it . As for the computation of $(\ord/\f_\ord)^*$, the computation is analogous to that of $(\ok/\f_\ord)^*$, except that we should compute the conductor $\f_\ord$ first. The whole algorithm has been described in \cite{picard_ord}. Note that for solving the DLP in the quotient of $(\ok/\f_\ord)^*$ modulo $(\ord/\f_\ord)^*$, we only need the algorithm to solve the DLP in the former group.

 Since $\varphi(\alpha)=\alpha\oo_{K,T}$ (or $\alpha\ok$) can be computed efficiently we can compute the cokernel of $\varphi$, that is the quotient group of $(\oo_{K,T}/\f)^*/(R/\f)^*$ modulo $\varphi(\oo_{K,T}^*)$, which is also equal to the kernel of $\phi$. Solving the DLP in the kernel of $\phi$ can be reduced to solving the DLP in $(\oo_{K,T}/\f)^*/(R/\f)^*$.

As for the fourth term in (\ref{exact seq.}), it can be computed by the following exact sequence from \cite{arith_nr}
$$ \bigoplus_{\p\in T}\zz \ra \Cl_K\ra\Cl(\oo_{K,T})\ra 1.$$
Computing the class group of $K$ and solving its DLP have been well-considered in many references and in Pari/GP \cite{gtm138}, {\ttfamily bnfinit} and {\ttfamily bnfisprincipal} respectively. The function {\ttfamily bnfsunit} can be used to compute the $T$-class group $\Cl(\oo_{K,T})$, but we fail to find the function to solve the DLP in it. Furthermore, for our final goal we need something more through computing the DLP in $\Cl(\oo_{K,T})$. Explicitly, suppose that $\Cl(\oo_{K,T})=\bigoplus_{i=1}^m \zz/d_i\zz\cdot[\pp_i]$\footnote{The bracket $[x]$ in a quotient group means a class represented by $x$. The meaning of ``class'' differ in different groups, but when there is no fear of confusion, we would use this symbol uniformly.}. For an ideal $\mathfrak{A}$ of $\oo_{K,T}$, we should compute not only the $(v_1,\cdots,v_m)$ such that $[\mathfrak{A}]=\prod [\pp_i]^{v_i}$, but also the $\gamma\in K$ such that $\mathfrak{A}=\gamma\oo_{K,T}\cdot\prod \pp_i^{v_i}$. The algorithm can be described as following:
\\
\\
\textbf{Algorithm 1.}

\begin{tabular}{ll}
  Input: & Class group $\Cl_K=(B,D_B)$ where $B=\{[\mathfrak{b}_1],\cdots,[\mathfrak{b}_l]\}$ and $D_B=\mbox{diag}(e_1,\cdots,e_l)$, \\
   & finite set of prime ideals $T=\{\p_1,\cdots,\p_t\}$, an ideal $I$ of $\ok$. \\
  Output: & $T$-class group $\Cl(\oo_{K,T})=(C,D_C)$ where $C=\{[\mathfrak{c}_1\oo_{K,T}],\cdots,[\mathfrak{c}_m\oo_{K,T}]\}$ and $\mathfrak{c}_i$ are\\
   & $\ok$-ideals and $D_C=\mbox{diag}(d_1,\cdots,d_m)$, $v=(v_1,\cdots,v_c)^t$ and $\gamma$ such that $I\oo_{K,T}=$\\
   & $\gamma\cdot \prod_{i=1}^m(\mathfrak{c}_i\oo_{K,T})^{v_i}$.
\end{tabular}
%
\begin{itemize}
\item[1.] Compute the discrete logarithm of every $[\p_i]$ in $\Cl_K$ as $[\p_i]=BX_i$. Set $X=[X_1,\cdots,X_t]$.
\item[2.] Set $M=[D_B|X]$. Compute the SNF of $M$: $UMV=[0|D'_C]$, and $C'=BU^{-1}=\{[\mathfrak{c}'_1],\cdots,[\mathfrak{c}'_l]\}$. Let $m$ be the largest index of $i$ s.t. the $(i,i)$-th component of $D'_C\neq 1$. Set $\mathfrak{c}_i=\mathfrak{c}'_i$, $C=\{[\mathfrak{c}_1\oo_{K,T}],\cdots,[\mathfrak{c}_m\oo_{K,T}]\}$ and $D_C$ to be the left upper $m\times m$ submatrix of $D_C'$.
\item[3.] Compute the discrete logarithm of $I$ in $\Cl_K$: $w=(w_1,\cdots,w_l)^t$ and $\alpha\in K$ such that $I=\alpha\prod_{i=1}^l\mathfrak{b}_i^{w_i}$. Compute $v'=(v'_1,\cdots,v'_l)^t=Uw$. For $1\leq i\leq l$, compute $v'_i=n_iD'_C[i,i]\footnote{For a matrix $M$, $M[i,j]$ denotes the $(i,j)$-th component of $M$.}+v_i$ with $0\leq v_i<v'_i$. Set $v=(v_1,\cdots,v_m)^t$.
\item[4.] Compute $\{\alpha_i\}_{i=1}^l$ such that $\mathfrak{b}_i^{e_i}=\alpha_i\ok$. For $1\leq i\leq l$, compute $\beta_i=\prod_{k=1}^l \alpha_k^{V[k,t+i]}$. Compute $\beta=\prod_{i=1}^l\beta_i^{n_i}$.
  Set $\gamma=\alpha\beta$.
\item[5.] Return $(C,D_C,v,\gamma)$.
\end{itemize}

\begin{remark}
Note that we use $\ok$-ideals $\{\mathfrak{c}_i\}$ to represent or store the generators of $\Cl(\oo_{K,T})$.
\end{remark}

\subsubsection{Computing $\Pic(R)$ and Solving the DLP in It.}
Now we come to the short exact sequence
\begin{equation}\label{exact seq.2}
  1\ra \frac{(\oo_{K,T}/\f)^*/(R/\f)^*}{\mbox{Im}(\varphi)}\xrightarrow{\psi} \Pic(R)\xrightarrow{\phi}\Cl(\oo_{K,T})\ra 1,
\end{equation}
where the first group and the last group have been computed, denoted by $\mathcal{A}=(\{[\alpha_i]\},D_A)$ and $\mathcal{C}=(\{[\mathfrak{c}_i\oo_{K,T}]\}_{i=1}^m,D_C)$ respectively. For computing $\Pic(R)$ first we consider how to represent an element of $\Pic(R)$.

  From Theorem \ref{main}, $\Pic(R)\cong\Pic(R,\f)$. Instead of dealing with an arbitrary ideal $\mathcal{I}$ of $R$, indeed we could choose another representative $\mathcal{I}'$ in the class of $\mathcal{I}$ which is also prime to $\f$. Note that $J(R,\f)\twoheadrightarrow \Pic(R,\f)$ is an epimorphism, and by the isomorphism $\iota: J(R,\f)\simeq J_K^{\tilde{\f}\m_0}$ from Corollary \ref{J(R,f)}, any element of $\Pic(R,\f)$ could be chosen as a class of an ideal of $R$ which is prime to $\f$, i.e. belongs to $J(R,\f)$. Then by $\iota$ this $R$-ideal can be replaced by an $\ok$-ideal which is prime to $\tilde{\f}\m_0$, e.g. $[\iota^{-1}(\mathfrak{a})]$ where $\mathfrak{a}$ is an ideal of $\ok$ prime to $\tilde{\f}\m_0$.



\begin{lemma}
In the sense of equalities in Lemma \ref{kernelphi}, all the elements $\{\alpha_i\}$ in $\mathcal{A}$ could be chosen from the group $(\ok/\f_\ord)^*$. Then $\psi([\alpha_i])=[\alpha_i\oo_{K,T}\cap R]$ where $\alpha_i\oo_{K,T}\cap R$ is automatically in $J(R,\f)$. If $\alpha_i\ok=\prod_{\p\in T}\p^{n_{\p}}\mathfrak{a}_i$ where $\mathfrak{a}_i$ is prime to $T$, then $\iota(\alpha_i\oo_{K,T}\cap R)=\mathfrak{a}_i$.
\end{lemma}
\begin{proof}
Note that $\f_\ord=\f\cap\ord=\tilde\f$. For any $\alpha$ in $(\oo_{K,T}/\f)^*$, its image of $\psi$ is the class of $\alpha\oo_{K,T}\cap R$. If $\alpha \in \ok$ prime to $\tilde{\f}$, i.e. $\alpha\ok+\tilde{\f}=\ok$, extending to $\oo_{K,T}$ then $\alpha\oo_{K,T}+\f=\oo_{K,T}$, which means that $\alpha\oo_{K,T}\in J(\oo_{K,T},\f)$. By Corollary \ref{J(R,f)} then $\alpha\oo_{K,T}\cap R\in J(R,\f)$. Assuming $\alpha\ok=\prod_{\p\in T}\p^{n_{\p}}\mathfrak{a}$ where $\mathfrak{a}$ is prime to $T$, then $\alpha\oo_{K,T}\cap R=\mathfrak{a}\oo_{K,T}\cap R$. Since $\alpha$ is prime to $\tilde{\f}$ so $\mathfrak{a}\in J_K^{\tilde{\f}\m_0}$. Hence $\iota^{-1}(\mathfrak{a})=\mathfrak{a}\oo_{K,T}\cap R$ which implies $\iota(\alpha\oo_{K,T}\cap R)=\mathfrak{a}$.
\end{proof}
This lemma gives us an easy way to choose a representative in $\psi([\alpha_i])$ that is an integral ideal of $\ok$ and prime to $\tilde{\f}\m_0$.
For finding such a representative in $\phi^{-1}([\mathfrak{c}_i\oo_{K,T}])$ for $\mathfrak{c}_i$ in $\mathcal{C}$, we only need to change the representative of $[\mathfrak{c}_i\oo_{K,T} ]$ by the following  modified ``weak approximation theorem'':

\begin{theorem}\cite[Proposition 1.3.8]{gtm193}\label{weak approx.}
Let $S$ be a finite set of prime ideals of $\ok$ and  $(e_\p)_{\p\in S}\in \zz^S$. There exists a polynomial-time algorithm that finds $a\in K$ such that $\emph{ord}_\p(a)=e_\p$ for all $\p\in S$ and $\emph{ord}_\p(a)\geq 0$ for all $\p\notin S$.
\end{theorem}
\noindent Thus we can replace $\mathfrak{c}_i$ by an integral $\ok$-ideal $\mathfrak{c}'_i$ prime to $\tilde\f\m_0$ in the same ideal class of $\mathfrak{c}_i$.
Then $\phi^{-1}([\mathfrak{c}'_i\oo_{K,T}])=[\mathfrak{c}'_i\oo_{K,T}\cap R]=[\iota^{-1}(\mathfrak{c}'_i)]$, and hence $\mathfrak{c}'_i$ is just the required one.

Now $\Pic(R)$ can be computed efficiently. For solving the DLP in $\Pic(R)$, it is a routine as explained at the beginning of this section. But there is an issue should be dealt with: how to recover the $\xi$ if given $\psi(\xi)$?
Assume the input is an ideal $\mathfrak{a}$ of $\ok$ which is prime to $\tilde{\f}\m_0$.
After computing the discrete logarithm of $\phi([\iota^{-1}(\mathfrak{a})])=\mathfrak{a}\oo_{K,T}$ by Algorithm 1, we get a $\gamma\in K$ and $v=(v_1,\cdots,v_c)^t$ such that $\mathfrak{a}\oo_{K,T}=\gamma\cdot \prod_{i=1}^m({\mathfrak{c}'}_i\oo_{K,T})^{v_i}$. If let $\mathfrak{d}=\mathfrak{a}\prod_{i=1}^m {\mathfrak{c}'}_i^{-v_i}$, then $\phi([\iota^{-1}(\mathfrak{d})])=[\gamma\oo_{K,T}]=[1]$. Hence $[\iota^{-1}(\mathfrak{d})]\in \mbox{Im}(\psi)$, and there is some $\xi\in (\oo_{K,T}/\f)^*$ such that $[\iota^{-1}(\mathfrak{d})]=[\xi\oo_{K,T}\cap R]$ in $\Pic(R)$ (also $[\gamma\oo_{K,T}]=[\xi\oo_{K,T}]$ in $\Cl(\oo_{K,T})$).
 Write $\gamma=\alpha/\beta$ where $\alpha,\beta\in \ok$. By Theorem \ref{weak approx.} we can find an $\eta\in K$ such that $\alpha\eta\in \ok$ and prime to $\tilde{\f}$ and $\beta\eta\in\ok$. Denoting $\alpha\eta$, $\beta\eta$ by $\alpha'$, $\beta'$ respectively, then $\gamma=\alpha'/\beta'$. Note that $\gamma\oo_{K,T}$ is prime to $\f$. Hence $\gamma\ok$ is prime to $\tilde{\f}$ and so is $\beta'\ok$. Then $\beta'\in (\ok/\tilde{\f})^*=(\oo_{K,T}/\f)^*$ and has inverse $\beta^{-1} \in (\oo_{K,T}/\f)^*$. Then $\xi=\alpha'\beta^{-1}$ is just what we need.

\subsection{Computing the Ring Class Field of $R$}\label{computation2}
Recall that $\tilde\f\m_0$ is an admissible modulus for the ring class field $H_R$ defined by $R$. And from Section $\S$\ref{theory} we have
\begin{equation}\label{final exact seq.}
  1\ra P^{\tilde{\f}\m_0}_{K,R}/P^{\tilde{\f}\m_0}_{K,1}\ra \Cl^{\tilde{\f}\m_0}_K \ra \Pic(R,\f)\ra 1.
\end{equation}
Pari/GP has function {\ttfamily bnrinit} to compute the ray class group $\Cl^{\tilde{\f}\m_0}_K$. Also we know how to compute the Picard group of $R$ and solve the DLP in it by previous argument. Then we can compute the congruence subgroup $ P^{\tilde{\f}\m_0}_{K,R}$ (modulo $P^{\tilde{\f}\m_0}_{K,1}$) determined by $\Pic(R,\f)$, which can be represented by a matrix $M_R$, c.f. \cite{Cohen_rcg}.

Finally, given a congruence subgroup $M_R$ for modulus $\tilde{\f}\m_0$, we should compute the corresponding class field $H_R$. It is a rather complicated and technical but interesting task. Pari/GP has function {\ttfamily rnfkummer} to compute it, but only in the case that the class field is a Kummer extension of prime degree over $K$. It uses a method based on ``Hecke's theorem'' \cite[Theorem 10.2.9]{gtm193}. For the general case, it is more avaliable to use a method based on ``Artin map'', which is proposed by Claus Ficker in 2000 \cite{Fieker}. For more knowledge one may refer to \cite{Cohen_ccft}.\\
\\
\emph{Example}~1. Let $K$ be an imaginary quadratic extension defined by $x^2+710=0$. Then $\text{d}(K)= -2840=-2^3\cdot5\cdot71$. Let $\alpha=\sqrt{-710}$ be a root of this definition equation. Then $\{1,\alpha\}$ is a set of integral basis of $K$. The class group $\Cl_K$ is equal to
$$\Cl_K=\zz/16\zz\cdot[\p_7] \oplus \zz/2\zz\cdot[\p_5] $$
where $7\ok=\p_7\tilde{\p}_7$ (split), $5\ok=\p_5^2$ (ramified). Let $a=7$, $T=\{\p_7,\tilde{\p}_7\}$, $R=\oo_{K,T}$. We get
$$\Pic(R)=\zz/2\zz\cdot[\pp_5]$$
where $\pp_5=\p_5\oo_{K,T}$, and the ring class field $H_R$ corresponding to $R$ has the form $K[x]/(x^2 - 2)$, which is equal to $\Q[x]/(x^4 + 354x^2 + 31684)$.

In this case, $R=\oo_{K,T}=\zz[\frac{1}{a}]\ok$. The norm form equation w.r.t. $R$ is $\textbf{\mbox{N}}(X+Y\alpha)=X^2 + 710Y^2$. The primes that are not belonging to $\set{2,5,7,19}$, smaller than $100$ and split in $H_R$ are $17,23,47,79,97$. We can find a solution $(X,Y)\in\zz$ such that $p\zz[\frac{1}{7}]^*=X^2 + 710Y^2$ for every $p\in\set{17,23,47,79,97}$:
\begin{eqnarray*}
  17\cdot 7^5 &=& 447^2+710\cdot11^2, \\
  23\cdot 7^7 &=& 1593^2+710\cdot152^2, \\
  47\cdot 7^3 &=& 69^2+710\cdot4^2, \\
  79\cdot 7^4 &=& 173^2+710\cdot15^2, \\
  97\cdot 7^7 &=& 1991^2+710\cdot327^2.
\end{eqnarray*}
It is obvious that for $p\in\set{17,23,47,79,97}$, $p=X^2+710Y^2$ has no integral solution . Therefore, for those $p$ splitting in $H_R$, it may happen that $p=X^2+710Y^2$ has no integral solution but $p\cdot 7^m=X^2+710Y^2$ must have an integral solution for some non-negative integer $m$.
\\
\\
\emph{Example}~2.  Here $K$ is a quartic extension over $\Q$ defined by $f= x^4 + 15x^2 + 45$. We get $\text{d}(K)=18000=2^4\cdot3^2\cdot5^3$. Let $\alpha$ be a root of $f$. Then $\set{1,\alpha,\beta:=1/3\alpha^2 + 2,\gamma:=1/3\alpha^3 + 3\alpha}$ is a set of integral basis of $K$. The class group $\Cl_K$ is equal to
$$\Cl_K=\zz/2\zz\cdot[\p_{11,1}]\oplus\zz/2\zz\cdot[\p_{2}]$$
where $11\ok=\prod_{i=1}^4\p_{11,i}$ and $2\ok=\p_{2}^2$. Let $a=11$ and $\ord=\zz[\alpha]$ and $R=\zz[\frac{1}{a}]\ord$. Here $[\ok:\ord]=9$. After computation we get
$$\Pic(R)=\zz/2\zz\cdot[\pp_{2}]$$
where $\pp_{2}=\p_{2}\oo_{K,T}\cap R$. Let $H_R$ be the ring class field defined by $R$. Its definition equation over $\Q$ is $x^8 + 40x^6 + 470x^4 + 1400x^2 + 25$.

The norm form equation w.r.t. $R$ is $\textbf{\mbox{N}}(X+Y\alpha+Z\beta+W\gamma)=$
 \begin{equation*}
   \begin{split}
     & X^4  - 30ZX^3 + (15Y^2 - 270WY + (315Z^2 + 1350W^2))X^2 + (-180ZY^2 + 2700WZY)X \\
       &  + (-1350Z^3 - 12150W^2Z)X + (45Y^4 - 1350WY^3 + (675Z^2 + 14175W^2)Y^2  \\
       & + (-8100WZ^2 - 60750W^3)Y + (2025Z^4 + 30375W^2Z^2 + 91125W^4))
   \end{split}
 \end{equation*}
 The primes that are not belonging to $\set{2,3,5,11}$, smaller than $150$ and split in $H_R$ are $61,71,131$. We can find a solution $(X,Y,Z,W)\in\zz$ such that $p\zz[\frac{1}{11}]^*=\textbf{\mbox{N}}(X+Y\alpha+Z\beta+W\gamma)$ for every $p\in\set{61,71,131}$:
 \begin{eqnarray*}
   61 &=& \textbf{\mbox{N}}(1+\alpha), \\
   71\cdot11 &=& \textbf{\mbox{N}}(1+2\alpha), \\
   131\cdot11^3 &=& \textbf{\mbox{N}}(2+8\alpha+3\beta+\gamma).
 \end{eqnarray*}

\end{document}